\documentclass[11pt,a4paper]{article}
\usepackage{amsmath,amsfonts,amssymb,amsthm}
\title{Typical curvature behaviour of bodies of constant width}
\renewcommand{\thefootnote}{\fnsymbol{footnote}}
\author{Imre B\'ar\'any and Rolf Schneider}
\date{}
\usepackage{amssymb}
\usepackage[all,cmtip]{xy}
\usepackage{xcolor}

\newtheorem{Proposition}{Proposition}[section]

\newtheorem{Lemma}[Proposition]{Lemma}
\newtheorem{Theorem}[Proposition]{Theorem}

\newcommand{\K}{\mathcal{K}}
\newcommand{\inn}{{\rm int}\,}

\newcommand{\R}{\mathbb{R}}
\newcommand{\Sn}{\mathbb{S}^{n-1}}

\newcommand{\Hc}{\mathcal{H}}

\newcommand{\bd}{{\rm bd}\,}

\begin{document}

\maketitle

\begin{abstract}
It is known that an $n$-dimensional convex body which is typical in the sense of Baire category, shows a simple, but highly non-intuitive curvature behaviour: at almost all of its boundary points, in the sense of measure, all curvatures are zero, but there is also a dense and uncountable set of boundary points at which all curvatures are infinite. The purpose of this paper is to find a counterpart to this phenomenon for typical convex bodies of given constant width. Such bodies cannot have zero curvatures. A main result says that for a typical $n$-dimensional convex body of constant width $1$ (without loss of generality), at almost all boundary points, in the sense of measure, all curvatures are equal to $1$. (In contrast, note that a ball of width $1$ has radius $1/2$, hence all its curvatures are equal to $2$.) Since the property of constant width is linear with respect to Minkowski addition, the proof requires recourse to a linear curvature notion, which is provided by the tangential radii of curvature. 

\vspace{2mm}

\noindent AMS 2010 {\it Mathematics subject classification}: Primary 52A20, Secondary 53A07, 54E52
\end{abstract}

\renewcommand{\thefootnote}{{}}
\footnote{Partially supported by ERC Advanced Research Grant no 267165 (DISCONV). The first
author was supported by Hungarian National Foundation Grant K 83767. }

\section{Introduction}

A general convex body in ${\mathbb R}^n$ shows, even without differentiability assumptions, a rather regular curvature behaviour at almost all of its boundary points. According to a well-known theorem of Aleksandrov \cite{Ale39}, almost all boundary points of a convex body, in the sense of $(n-1)$-dimensional Hausdorff measure $\Hc^{n-1}$, are normal points. At a normal point, all sectional curvatures exist, and they satisfy the theorems of Euler and Meusnier. From the generic viewpoint, the picture becomes partially simpler, partially more irregular. While a typical convex body, in the Baire category sense (the definition will be recalled in Section \ref{sec2}), is strictly convex and smooth (its boundary is $C^1$), it was proved by Zamfirescu \cite{Zam80} that at almost all of its boundary points the curvatures are zero. This was recently \cite{Sch14} supplemented by the observation that the boundary of a typical convex body contains an uncountable, dense set of points where all curvatures are infinite, and the set of these points has a spherical image of full $\Hc^{n-1}$ measure in the unit sphere $\Sn$.

An interesting and intensively studied subclass of the convex bodies are the bodies of constant width. It suffices to consider those of constant width $1$. A convex body $K$ has constant width $1$ if each two distinct parallel supporting hyperplanes of $K$ have distance $1$, equivalently, if the Minkowski sum of $K$ and its reflected image $-K$ is the unit ball. It follows from the latter fact that $0\le \varrho\le 1$ if $\varrho$ is a principal radius of curvature of $K$, taken at the (unique) boundary point with a given outer normal vector $u$. As shown by Aleksandrov \cite{Ale39}, the radii of curvature exist for $\Hc^{n-1}$-almost all unit vectors $u\in\Sn$. In analogy to the Baire category type results for general convex bodies, it is to be expected that for a typical convex body of constant width $1$, the radii of curvature exhibit a preference for attaining the values $0$ and $1$. Indeed, it was shown by Zamfirescu \cite{Zam94} that a for a typical convex domain of constant width $1$ in the plane, the radii of curvature attain only the values $0$ and $1$. 

The first of the following theorems can be regarded as a higher-dimensional extension of this result.

\begin{Theorem}\label{T1}
A typical convex body $K$ of constant width $1$ in $\R^n$ has the property that for $\Hc^{n-1}$-almost all $u\in \Sn$, either all radii of curvature of $K$ at the normal vector $u$ are equal to $1$ or at least one radius of curvature at $u$ is equal to $0$.
\end{Theorem}

It makes a difference whether we consider radii of curvature at given normal vectors or at given boundary points.

\begin{Theorem}\label{T2}
A typical convex body $K$ of constant width $1$ in $\R^n$ has the property that for $\Hc^{n-1}$-almost all $x\in \bd K$, all radii of curvature of $K$ at $x$ are equal to $1$. 

Moreover, $K$ has an uncountable, dense set of boundary points at which all radii of curvature are zero.
\end{Theorem}

In the plane, the Reuleaux polygons of width $1$ are dense in the set of all convex bodies of constant width $1$, and they have the property that their radius of curvature is everywhere either $0$ or $1$. This makes them a useful tool for the proof of the mentioned Baire category type result in the plane. In higher dimensions, there is no similar construction known. Therefore, the proof of the preceding theorems requires a more elaborate approximation procedure. 

After collecting some notation and basic notions in the next section, we use Section \ref{sec3} to explain the tangential radii of curvature, which are fundamental for our investigation, and to provide some results pertaining to them which will be needed later. Section \ref{sec4} prepares the proof of an approximation theorem for bodies of constant width, which is crucial for the proof of the main results; this theorem is then proved in Section \ref{sec4.1}.
In Section \ref{sec5} we prove a more general version of Theorem \ref{T1}, involving tangential radii of curvature, from which Theorems \ref{T1} and \ref{T2} will then be deduced.

\section{Preliminaries}\label{sec2}

We work in ${\mathbb R}^n$ with scalar product $\langle\cdot,\cdot\rangle$ and induced norm $\|\cdot\|$. We write $B(z,r):=\{x\in{\mathbb R}^n: \|x-z\|\le r\}$ for the closed ball and ${\mathbb S}(z,r):=\{x\in{\mathbb R}^n: \|x-z\|= r\}$ for the sphere with centre $z$ and radius $r>0$. As usual, we write $B(o,1)=B^n$ and ${\mathbb S}(o,1)=\Sn$, where $o$ denotes the origin of $\R^n$.

The distance of a point $z$ from a set $A$ will be denoted by ${\rm dist}\,(z,A)$.

${\mathcal K}^n$ is the set of convex bodies (nonempty, compact, convex subsets) of ${\mathbb R}^n$. It is equipped with the Hausdorff metric, denoted by $\delta$. The support function of the convex body $K\in{\mathcal K}^n$ is defined by $h(K,u)=\max\{\langle x,u\rangle:x\in K\}$ for $u\in\R^n$. A convex body $K$ is {\em of constant width} $d>0$ if $h(K,u)+h(K,-u)=d$ for all $u\in \Sn$. We denote by ${\mathcal W}^n_1$ the set of convex bodies of constant width $1$ in ${\mathbb R}^n$. We recall that every convex body $K$ of constant width is strictly convex.

A {\em support element} of the convex body $K$ is a pair $(x,u)$ where $x$ is a boundary point of $K$ and $u$ is an outer unit normal vector to $K$ at $x$. The boundary point $x$ is {\em smooth} or {\em regular} if at $x$ there is only one outer unit normal vector. The vector $u\in \Sn$ is a {\em regular} normal vector of $K$ if it is a normal vector of $K$ at a unique boundary point of $K$, which is then denoted by $x_K(u)$. The support element $(x,u)$ is called {\em regular} if $x$ and $u$ are both regular. If a convex body $K$ is strictly convex, then each $u\in\Sn$ is a regular normal vector of $K$. The supporting hyperplane of $K$ with outer normal vector $u$ is denoted by $H(K,u)$, and $F(K,u)=H(K,u)\cap K$ is the corresponding support set. The $(n-1)$-dimensional subspace $u^\perp= T(K,u)$ parallel to $H(K,u)$ is called the {\em tangent space} of $K$ at $u$. If $x$ is a smooth boundary point of $K$ and $u$ is the unique outer unit normal vector of $K$ at $x$, then we write $T(K,u)=T_xK$.

We refer to \cite{Sch14} for notions from the theory of convex bodies that are not explained here.

We recall the notion of a `typical convex body'. The space $\K^n$ with the Hausdorff metric, being a complete metric space, is a Baire space. A subset of such a space is called {\em meagre} if it is a countable union of nowhere dense sets, and {\em comeagre} (or {\em residual}\hspace{2pt}) if its complement is meagre. The essential property of a Baire space is that the intersection of countably many comeagre subsets is still dense; therefore, comeagre subsets of a Baire space deserve the attribute of being `large'. One says that a {\em typical} $n$-dimensional convex body has a property $P$ (or that {\em most} bodies in $\K^n$ have property $P$) if the set of convex bodies having property $P$ is comeagre in $\K^n$. The same notions can be applied to the space of convex bodies in $\R^n$ of constant width $1$, since also this space, with the Hausdorff metric, is a complete metric space.

\section{Tangential radii of curvature}\label{sec3}

The defining relation for a convex body $K$ to be of constant width $1$ is the equation $h(K,u)+h(K,-u)=1$ for $u\in \Sn$. In order to be able to exploit this for curvature investigations, we need a curvature notion for convex bodies that depends linearly on the support function. Tangential radii of curvature are appropriate in this respect. We introduce them in this section, study their relation to the usual sectional curvatures, and prove some results that will be needed in the proofs of our theorems. Since there seems to be no convenient reference for tangential radii of curvature, we give a slightly extended introduction.

We begin with a general notion of curvature of a planar convex body and follow Jessen \cite{Jes29} in its introduction. Let $n=2$. Let $K\in{\mathcal K}^2$ be a two-dimensional convex body and let $(x,u)$ be a support element of $K$. Let $t\perp u$ be a unit vector. We denote by $R_{-u}$ and $R_t$ the closed rays with endpoint $x$ spanned by $-u$ and $t$, respectively, and define the open halfplane
$$ E:= \{x+\lambda u +\mu t: \lambda\in{\mathbb R},\,\mu>0\}.$$
We assume that the set $E\cap\bd K$ is not empty and also assume at first that $R_t \cap  E\cap\bd K=\emptyset$. Let $z\in E\cap{\rm bd}\,K$. Let $r(z)$ be the radius of the unique circle $C(z)$ with centre on $R_{-u}$ that passes through $x$ and $z$. Any limit circle (possibly degenerate) of a sequence of circles $C(z)$ for a sequence of points $z$ converging to $x$ is called an {\em osculating circle} of $K$ at $(x,u)$ in direction $t$. The radii of all osculating circles in direction $t$ fill a closed interval in $[0,\infty]$, which we denote by
$$ [\varrho_i(K,x,u,t),\varrho_s(K,x,u,t)].$$
If (what was excluded at first) there is a point $z\in  R_t\cap E\cap{\rm bd}\,K$, then the halfplane $\{y\in\R^2: \langle y-x,u\rangle\le 0\}$ is considered as a degenerate osculating circle of $K$ at $(x,u)$ in direction $t$, and we put $r(z)=\infty$ and $\varrho_i(K,x,u,t) = \varrho_s(K,x,u,t)=\infty$. Thus, in either case we have
$$ \varrho_i(K,x,u,t) =\liminf_{z\to x} r(z),\qquad \varrho_s(K,x,u,t) =\limsup_{z\to x} r(z).$$
The (extended real) numbers $\varrho_i(K,x,u,t),\varrho_s(K,x,u,t)$ are called, respectively, the \emph{lower radius of curvature} and the  \emph{upper radius of curvature}  of $K$ at $(x,u)$ in direction $t$. If $\varrho_i(K,x,u,t)=\varrho_s(K,x,u,t)$, then this (extended real) number is denoted by $\varrho(K,x,u,t)$ and called the \emph{radius of curvature} of $K$ at $(x,u)$ in direction $t$. If $K$ is clear from the context, we simply write $\varrho_i(x,u,t)$, $\varrho_s(x,u,t)$, $\varrho(x,u,t)$.

There is a second way of introducing radii of curvature, using normal vectors instead of boundary points. Again we assume that $E\cap{\rm bd}\,K\not=\emptyset$. Let $w$ be a unit vector with $\langle w,t\rangle>0$. At first, we assume that the support line $H(K,w)$ does not pass through $x$. Let $r(w)$ be the radius of the unique circle $C'(w)$ with centre on $R_{-u}$ that passes through $x$ and is tangent to $H(K,w)$. Any limit circle (possibly degenerate) of a sequence of circles $C'(w)$ for a sequence of vectors $w$ (with $\langle w,t\rangle>0$) converging to $u$ is called (for the moment) an {\em osculating circle of the second kind} of $K$ at $(x,u)$ in direction $t$. In contrast, we say that the previously defined osculating circles are {\em of the first kind}. If $x\in H(K,w)$ for some $w$ with $\langle w,t\rangle>0$, then $\{x\}$ is considered as a degenerate osculating circle of the second kind of $K$ at $(x,u)$ in direction $t$. There is, however, no difference between osculating circles of the first and the second kind. Consider, for example, a circle $C'(w)$ as just constructed. Then $H(K,w)$ is a common support line of $K$ and $C'(w)$. Let $p$ be a point where $H(K,w)$ touches $K$ and let $q$ be a point where $H(K,w)$ touches $C'(w)$. Then at least one of the following occurs: the boundary arc of $K$ between $p$ and $x$ has a point $z$ in common with $C'(w)$, or the arc of $C'(w)$ between $q$ and $x$ has a point $z$ in common with the boundary of $K$. Therefore, the circle $C'(w)$ coincides with a circle $C(z)$ constructed in the first procedure. If a sequence of vectors $w$ converges to $u$, the corresponding sequence of points $z$ converges to $x$ or to a point of $R_t\cap E\cap\bd K$. We conclude that every osculating circle of the second kind is also one of the first kind.  A similar argument (roughly, interchanging the roles of boundary points and support lines) shows that every osculating circle of the first kind is also one of the second kind. Thus, we need not distinguish between the two kinds, and the lower and upper radii of curvature can be defined in either of the two ways. 

We still assume that $(x,u)$ is a support element of $K$. If $x$ is a smooth boundary point of $K$, then the outer unit normal vector $u$ of $K$ at $x$ is uniquely determined, and we briefly write $\varrho_i(x,t)$, $\varrho_s(x,t)$, $\varrho(x,t)$ for the introduced numbers. The reciprocal values (possibly $\infty$)
$$\kappa_{i}(x,t):= 1/\varrho_{s}(x,t), \qquad \kappa_{s}(x,t):=1/\varrho_{i}(x,t),  \qquad \kappa(x,t):=1/\varrho(x,t)$$
are, respectively, the \emph{lower curvature}, \emph{upper curvature} and, if it exists, the \emph{curvature} of $K$ at $x$ in direction $t$. If $u$ is a regular normal vector of $K$, then the point $x$ in $F(K,u)$ is uniquely determined, and we write $\varrho_i(u,t)$, $\varrho_s(u,t)$, $\varrho(u,t)$.

At regular boundary points or normal vectors, the radii of curvature can be expressed in terms of functions locally representing the convex body. Let $K\in{\mathcal K}^2$ and a support element $(x,u)$ of $K$ be given, and assume that $x$ is a smooth boundary point of $K$. Let $t\perp u$ be a unit vector. We can choose a number $\varepsilon>0$ and a neighbourhood $U$ of $x$ such that $U\cap\bd K$ can be described in the form
\begin{equation}\label{2.5.1}
U\cap\bd K=\big\{z+\tau t-f(\tau)u:\tau\in(-\varepsilon,\varepsilon)\big\},
\end{equation}
where $f:(-\varepsilon,\varepsilon)\to {\mathbb R}$ is a convex function which satisfies $f\ge 0$, $f(0)=0$ and $f'(0)=0$. An elementary calculation of the radii of the circles $C(z)$ used in the definition of osculating circles then gives
\begin{equation}\label{2.2}
\varrho_i(K,x,t) =\liminf_{\tau\downarrow 0} \frac{\tau^2}{2f(\tau)},\qquad \varrho_s(K,x,t) =\limsup_{\tau\downarrow 0} \frac{\tau^2}{2f(\tau)}.
\end{equation}

\vspace{2mm}

Instead of using a local representation (\ref{2.5.1}) of ${\rm bd}\, K$, we can also represent $K$ by its support function $h(K,\cdot)$. Then, under the assumption that $u$ is a regular normal vector of $K$ and $x=x_K(u)$, the second approach above (for introducing radii of curvature), together with an elementary calculation, shows that, for a unit vector $t\perp u$,
\begin{equation}\label{2.2.a} 
\varrho_i(K,u,t) =\liminf_{w\to u,\, \langle w,t\rangle>0} \frac{h(K,w)-\langle x,w\rangle}{1-\langle u,w\rangle},
\end{equation}
\begin{equation}\label{2.2.b} 
\varrho_s(K,u,t) =\limsup_{w\to u,\, \langle w,t\rangle>0} \frac{h(K,w)-\langle x,w\rangle}{1-\langle u,w\rangle}.
\end{equation}

\vspace{2mm}

Now we turn to general dimensions $n\ge 2$. Let $K\in{\mathcal K}^n$ be an $n$-dimensional convex body, let $(x,u)$ be a support element of $K$ and let $t\perp u$ be a unit vector. The notion of curvature of $K$ at $(x,u)$ in direction $t$ can be reduced to the two-dimensional case, either by section or by projection. We consider first the case of sections, which leads to sectional curvatures. 

For this, we define the two-dimensional subspace
$$ L_{x,u,t} := x+{\rm lin}\{u,t\} $$
and the two-dimensional convex body
%\begin{equation}\label{2.32a}
$$ K_{x,u,t}:=K \cap L_{x,u,t}$$
%\end{equation}
and then simply define
%\begin{equation}\label{2.32b} 
$$ \varrho_i(K,x,t):= \varrho_i(K_{x,u,t},x,u,t),$$
%\end{equation}
and similarly for $\varrho_s$, $\varrho$ and the curvatures. We shall talk of {\em sectional radii of curvature} and {\em sectional curvatures}, but often delete `sectional' if there is no danger of ambiguity. The convention about dropping the argument $K$ is as in the two-dimensional case.

Sectional curvatures at smooth boundary points can be interpreted with the aid of the curvature indicatrix, which we now introduce. Let $x$ be a smooth boundary point of $K$ and $u$ the unique outer normal vector of $K$ at $x$. For  $\eta>0$ we consider the set
\begin{equation}\label{2.5.3d} 
{\mathcal D}_x(K,\eta):= (2\eta)^{-1/2}  \left\{[K \cap (H(K,u)-\eta u)]+\eta u-x\right\},
\end{equation}
which lies in the tangent space $T_{\hspace*{-1pt}x}K$ and contains $o$ if $\eta$ is sufficiently small (what we assume in the following). It is a suitable homothet of the intersection of $K$ with a hyperplane parallel to the supporting hyperplane at $x$ and at distance $\eta$ from it.
If the radial limit 
\begin{equation}\label{2.5.3c} 
\text{r-}\lim_{\eta\downarrow 0} {\mathcal D}_x(K,\eta) ={\mathcal D}_x(K) 
\end{equation}
exists, it is called the {\em curvature indicatrix}, or briefly the {\em indicatrix}, of $K$ at $x$. Here (\ref{2.5.3c}) means that ${\mathcal D}_x(K)$ is a closed set in $T_xK$ containing $o$ and 
$$ \lim_{\eta\downarrow 0} \rho({\mathcal D}_x(K,\eta),t)= \rho({\mathcal D}_x(K),t)\qquad\mbox{for } t\in T_{\hspace*{-1pt}x}K\setminus\{o\}$$ 
where $\rho$ denotes the radial function (as defined in \cite[p. 57]{Sch14}). The set ${\mathcal D}_x(K)$ is then necessarily convex.

Let $t\in T_xK\cap {\mathbb S}^{n-1}$. If we represent $L_{x,u,t}\cap{\rm bd}\,K$ in a neighbourhood of $x$ in the form  (\ref{2.5.1}), then the number $\tau>0$ for which $f(\tau)=\eta$ satisfies
$$ \frac{\tau}{\sqrt{2\eta}} =\rho({\mathcal D}_x(K,\eta),t),$$
hence
(\ref{2.2}) gives
\begin{equation}\label{9a} 
\varrho_i(x,t)=\liminf_{\eta\downarrow 0} \rho^2({\mathcal D}_x(K,\eta),t),\qquad \varrho_s(x,t)= \limsup_{\eta\downarrow 0} \rho^2({\mathcal D}_x(K,\eta),t).
\end{equation}
Therefore, if the radius of curvature in direction $t$ exists, then
\begin{equation}\label{2.5.3a} 
\varrho(x,t)=\lim_{\eta\downarrow 0} \rho^2({\mathcal D}_x(K,\eta),t),
\end{equation}
and if it exists for all $t$, then
\begin{equation}\label{2.80a}  
\varrho(x,t) = \rho^2({\mathcal D}_x(K),t) \qquad\mbox{for }t\in T_{\hspace*{-1pt}x} K \cap {\mathbb S}^{n-1}.
\end{equation}

\vspace{1mm}

A second kind of radii of curvature, called {\em tangential}, can be introduced by using projections. Let $K\in{\mathcal K}^n$. We assume now that $u$ is a regular unit normal vector of $K$, and we write $x_K(u)=x$. We consider the two-dimensional convex body
$$\overline K_{x,u,t}:= K|L_{x,u,t}.$$
where $|$ denotes orthogonal projection, and then define 
$$ \bar\varrho_i(K,u,t) := \varrho_i(\overline K_{x,u,t},x,u,t),$$
and similarly $\bar\varrho_s$ and $\bar\varrho$ are defined. We call these the (lower, etc.) {\em tangential radii of curvature} of $K$ at $u$ in direction $t$. Again, we drop the argument $K$ if there is no danger of ambiguity.

For later use, we observe that from $K\cap L_{x,u,t}\subset K|L_{x,u,t}$ it follows easily that at regular support elements $(x,u)$ of $K$ we have
\begin{equation}\label{3.19}
\varrho_s(K,x,t) \le \bar\varrho_s(K,u,t).
\end{equation}

We represent $K|L_{x,u,t}$ in a neighbourhood of $x$ in the form (\ref{2.5.1}). Let $\tau>0$ be the number for which $f(\tau)=\eta$. The point $\bar y:= x+\tau t-\eta u$ is the image of a point $y\in{\rm bd}\,K$ under the orthogonal projection, $\Pi$, from $\R^n$ to $L_{x,u,t}$. The $(n-2)$-plane $\Pi^{-1}(\bar y)$ lies in the $(n-1)$-dimensional space $H(K,u)-\eta u$ and is, in this space, a supporting plane of $K\cap (H(K,u)-\eta u)$ at the point $y$. Therefore, the distance of $\bar y$ from the line $x+\R u$, which is $\tau$, is equal to the value of the support function of $[K\cap(H(K,u)-\eta u)]+\eta u-x$ at $t$. According to (\ref{2.5.3d}), this yields
$$ \frac{\tau}{\sqrt{2\eta}} = h({\mathcal D}_x(K,\eta),t)$$
(recall that $h$ denotes the support function). Consequently, (\ref{2.2}) gives  
\begin{equation}\label{11a}
\bar\varrho_i(u,t)=\liminf_{\eta\downarrow 0} h^2({\mathcal D}_x(K,\eta),t),\qquad \bar\varrho_s(u,t)= \limsup_{\eta\downarrow 0} h^2({\mathcal D}_x(K,\eta),t).
\end{equation}
Hence, if the tangential radius of curvature in direction $t$ exists, then
\begin{equation}\label{2.5.3b} 
\bar\varrho(u,t)=\lim_{\eta\downarrow 0} h^2({\mathcal D}_x(K,\eta),t).
\end{equation}
If the tangential radius of curvature and thus the limit (\ref{2.5.3b}) exists for all $t$, then there is a closed convex set $\overline{\mathcal D}_u(K)\subset u^\perp$ with 
\begin{equation}\label{2.79d}
\text{s-}\lim_{\eta\downarrow 0}{\mathcal D}_x(K,\eta)=\overline{\mathcal D}_u(K)
\end{equation} 
(this convergence is defined by pointwise convergence of support functions)
and thus
\begin{equation}\label{2.81a}  
\bar\varrho(u,t) = h^2(\overline{\mathcal D}_u(K),t) \qquad\mbox{for }t\in u^\perp \cap {\mathbb S}^{n-1}.
\end{equation}
We call $\overline{\mathcal D}_u(K)$ the {\em tangential curvature indicatrix} of $K$ at $u$. In general, one cannot deduce the existence of $\overline{\mathcal D}_u(K)$ from the existence of ${\mathcal D}_x(K)$, nor vice versa, but the following holds.

\begin{Lemma}\label{L0} Let $(x,u)$ be a regular support element of $K\in{\mathcal K}^n$. Then the following conditions $\rm (a)$ and  $\rm (b)$ are equivalent.\\[1mm]
$\rm (a)$ The sectional radii of curvature of $K$ at $x$ exist and are positive and finite.\\[1mm]
$\rm (b)$ The tangential radii of curvature of $K$ at $u$ exist and are positive and finite.\\[1mm]
If $(a)$ or $(b)$ holds, then ${\mathcal D}_x(K)=\overline{\mathcal D}_u(K)$.
\end{Lemma}

\noindent{\em Proof.} By (\ref{9a}) and (\ref{2.80a}), property (a) holds if and only if $\lim_{\eta\downarrow 0} \rho({\mathcal D}_x(K,\eta),\cdot)$ exists pointwise and is positive and bounded. Similarly, by (\ref{11a}) and (\ref{2.81a}), property (b) holds if and only if  $\lim_{\eta\downarrow 0} h({\mathcal D}_x(K,\eta),\cdot)$ exists pointwise and is positive and bounded. Therefore, the assertion of the lemma follows from Lemma \ref{L7.1} in the Appendix. \qed

We shall need some second-order differentiability properties of the boundaries of convex bodies, which in dimension $n=3$ were established by Busemann and Feller \cite{BF36} and for general $n$ by Aleksandrov \cite{Ale39}. A convex function $f:D\to\R$, where $D\subset \R^n$ is open and convex, is said to have a strong second-order Taylor expansion at $x\in D$ if $f$ is differentiable at $x$ and there exists a symmetric linear map $A:\R^n\to\R^n$ such that
$$ f(y)= f(x)+\ \langle \nabla f(x),y-x\rangle+\frac{1}{2}\langle A(y-x),y-x\rangle +o(\|y-x\|^2)$$
for $y\in D$ and $y\to x$ (where $\nabla$ denotes the gradient). If this holds, then $A$ is uniquely determined, and $f$ is said to be twice differentiable at $x$. The theorem of Busemann--Feller--Aleksandrov says that $f$ is twice differentiable at $\Hc^n$-almost every point in $D$. For a modern proof, we refer to Borwein and Vanderwerff \cite[Thm. 2.6.4]{BV10}.

Let $K\in\K^n$ be a convex body with interior points and let $x\in\bd K$. The point $x$ is said to be a {\em normal point} of $K$ if at $x$ all sectional curvatures exist (in ${\mathbb R}$) and the function $q$ defined by 
\begin{equation}\label{3.25}
q(t):=\rho^{-2}({\mathcal D}_x(K),t)=\kappa(K,x,t), \qquad t\in T_xK\cap\Sn,
\end{equation}
(see (\ref{2.80a}) for the last equality) and homogeneous extension of degree two is a quadratic form. By applying the Busemann--Feller--Aleksandrov theorem to convex functions locally representing the boundary of $K$, it can be deduced that $\Hc^{n-1}$-almost all boundary points of $K$ are normal (Aleksandrov \cite{Ale39}, see also Busemann \cite{Bus58}). If (\ref{3.25}) holds with a quadratic form $q$, then (by transforming $q$ to principal axes) there is an orthornormal basis $(e_1,\dots,e_{n-1})$ of $T_xK$ and there are numbers $k_1,\dots,k_{n-1}$ such that
\begin{equation}\label{3.26}
\kappa(K,x,t) =k_1\langle t,e_1\rangle^2+\dots+k_{n-1}\langle t,e_{n-1}\rangle^2 \qquad\text{for }t\in T_xK\cap\Sn.
\end{equation}
This is known as (the generalized version of) Euler's formula. The numbers $k_1,\dots,k_{n-1}\ge 0$ are called the principal curvatures of $K$ at $x$.

Similarly, a unit vector $u$ is called an {\em ordinary normal vector} of $K$ if all tangential radii of curvature of $K$ at $u$ exist (in ${\mathbb R}$) and the function $\bar q$ defined by
\begin{equation}\label{3.27}
\bar q(t):=h^2(\overline{\mathcal D}_u(K),t)=\bar\varrho(K,u,t), \qquad t\in u^\perp\cap\Sn,
\end{equation}
(see (\ref{2.81a}) for the last equality) and homogeneous extension of degree two is a quadratic form. By applying the theorem on the twice differentiability of convex functions to the support function, Aleksandrov \cite{Ale39} has deduced (though not with this terminology) that $\Hc^{n-1}$-almost all $u\in\Sn$ are ordinary normal vectors of $K$. If (\ref{3.27}) holds with a quadratic form $\bar q$, then there is an orthornormal basis $(e_1,\dots,e_{n-1})$ of $u^\perp$ and there are numbers $r_1,\dots,r_{n-1}$ such that
\begin{equation}\label{3.28}
\bar\varrho(K,u,t) =r_1\langle t,e_1\rangle^2+\dots+r_{n-1}\langle t,e_{n-1}\rangle^2 \qquad\text{for }t\in u^\perp\cap\Sn.
\end{equation}
This is sometimes called Blaschke's formula, since a special case appears in \cite[\S24, III]{Bla56}. The numbers $r_1,\dots,r_{n-1}\ge 0$ are called the principal tangential radii of curvature of $K$ at $u$.

Suppose now that $(x,u)$ is a regular support element of $K$, $u$ is an ordinary normal vector of $K$ and (\ref{3.28}) holds with $r_1,\dots,r_{n-1}>0$. Then it follows from Lemma \ref{L0} that $x$ is a normal point of $K$ and that
${\mathcal D}_x(K)=\overline{\mathcal D}_u(K)$. In particular, we can choose $(e_1,\dots,e_{n-1})$ as the same orthonormal basis in (\ref{3.26}) and (\ref{3.28}). The curvature indicatrix ${\mathcal D}_x(K)$ is an ellipsoid, and $\pm e_1,\dots,\pm e_{n-1}$ are the directions of its principal axes. Therefore, $\rho({\mathcal D}_x(K),e_i)=h({\mathcal D}_x(K),e_i)$ and hence $r_i=1/k_i$ for $i=1,\dots,n-1$. Thus, the principal tangential radii of curvature at $u$ are equal to the principal sectional radii of curvature, that is, the reciprocal sectional curvatures, at $x$ (but each tangential radius of curvature in a non-principal direction is strictly larger than the sectional radius of curvature in the same direction).

\section{An approximation result for bodies of constant width}\label{sec4}

In this section, we formulate an approximation result for bodies of constant width. It serves as a higher-dimensional substitute for the denseness of the set of Reuleaux polygons in the set of planar convex bodies of constant width. The particular curvature properties that we need are only obtained on part of the boundary, but even this restricted approximation result will be sufficient for our purpose.

For $z\in {\mathbb S}^{n-1}$ we define the open halfsphere
\begin{equation}\label{C0} 
S(z) :=  \{u\in {\mathbb S}^{n-1}: \langle u,z\rangle >0\}
\end{equation}
(to be used later) and the closed spherical cap
\begin{equation}\label{C1}
C(z):= \{u\in {\mathbb S}^{n-1}: \langle u,z\rangle \ge 1/2\}
\end{equation}
(the number $1/2$ could be replaced by any positive number less than $1$).

\begin{Theorem}\label{T4}
Let $W_0\in{\mathcal W}^n_1$ and $v\in {\mathbb S}^{n-1}$ be given, let $\varepsilon>0$. Then there exists a convex body $Q\in{\mathcal W}^n_1$ with 
$$ \delta(W_0,Q)\le \varepsilon$$
and having the following property. For each $u\in C(v)$, either
\begin{equation}\label{3.2} 
\bar\varrho(Q,u,t) =1 \quad \mbox{for all } t\in u^\perp\cap \Sn
\end{equation}
or
\begin{equation}\label{3.1} 
\bar\varrho(Q,u,t) =0 \quad  \mbox{for at least one } t\in u^\perp\cap \Sn.
\end{equation}
\end{Theorem}

The proof of this theorem requires a number of preparations.

Let $K\in{\mathcal K}^n$ be a convex body. We denote by $d_K={\rm diam}\,K$ its diameter and by $r_K$ its inradius. Any closed segment $[x,y]\subset K$ with $\|x-y\|=d_K$ is called a {\em diameter segment} of $K$. A boundary point $x$ of $K$ is a {\em diameter endpoint} if it is an endpoint of a diameter segment of $K$. If each boundary point of $K$ is a diameter endpoint, then $K$ is of constant width $d_K$; this is well known and easy to see. 

A convex body $K'$ is called a {\em tight cover} of $K$ if it contains $K$ and has the same diameter as $K$.

According to \cite[(7.3)]{CG83}, every convex body $W\in{\mathcal W}^n_1$ contains a ball of radius $c_n>0$, where $c_n$ is an explicit constant that depends only on the dimension. We define
$${\mathcal K}^\bullet:= \{K\in{\mathcal K}^n: d_K=1, \, r_K\ge c_n\}.$$

Let $K\in{\mathcal K}^n$ and $\omega\subset \Sn$. The set $\tau(K,\omega)$ is the reverse spherical image of $K$ at $\omega$ (see \cite[p. 88]{Sch14}), that is, the set of points $x\in{\rm bd}\,K$ at which there exists an outer normal vector falling in $\omega$. If $K$ is strictly convex, then $\tau(K,\omega)= \{x_K(u):u\in\omega\}$. Let $v\in {\mathbb S}^{n-1}$. Then $\tau(K,S(v))$ is the set of points $x\in{\rm bd}\,K$ at which there exists an outer normal vector $u$ with $\langle u,v\rangle>0$. We call $\tau(K,S(v))$ the {\em upper boundary} of $K$, with respect to $v$; correspondingly, $\tau(K,S(-v))$ is the {\em lower boundary} of $K$.

In addition, let $L\in {\mathcal K}^n$ and $K\subset L$. A point $x\in{\rm bd}\,L$ is called $(K,v)${\em -directed} if $x+\lambda v\notin L$ for all $\lambda>0$ and $x-\lambda v\in{\rm int}\,K$ for some $\lambda>0$. The body $L$ is called $(K,v)${\em -complete} if each of its $(K,v)$-directed boundary points is a diameter endpoint of $L$. 

Further, we define
$$ Z^+(K,v) := \{x+\lambda v:x\in K,\,\lambda\ge 0\},$$
which is the cylinder above $K$ in direction $v$, and
$$ \beta_v(K):=Z^+(K,v) \cap \bigcap_{x\in K} B(x,d_K).$$
The following was proved in \cite{MS13} (Lemma 4; the proof holds for general convex bodies $K$). 

\begin{Lemma}\label{L1} 
Let $K\in {\mathcal K}^n$ and $v, w\in {\mathbb S}^{n-1}$. Then $\beta_v(K)$ is a tight cover of $K$, and it is $(K,v)$-complete.

Moreover, the body $\beta_w(\beta_v(K))$ is $(K,v)$-complete and $(K,w)$-complete. Every diameter endpoint of $\beta_v(K)$ is a diameter endpoint of $\beta_w(\beta_v(K))$.
\end{Lemma}

The last assertion (in \cite{MS13} only mentioned in the proof of Lemma 4) is clear, since $\beta_w(\beta_v(K))$ is a tight cover of $\beta_v(K)$.

The next lemma introduces the `generalized B\"uckner completion' (now denoted by $\beta$), which was established in \cite{MS13} for finite-dimensional normed spaces. We require it here only for Euclidean spaces (where some of the estimates could be improved, but that is irrelevant).  

\begin{Lemma}\label{L2} 
There are a number $p\in{\mathbb N}$, vectors $v_1,\dots,v_p\in {\mathbb S}^{n-1}$ and constants $\varepsilon_0>0$, $\ell>0$, all depending only on the dimension, such that the following holds. 

Let $K\in{\mathcal K}^\bullet$ and define
$$ K_0:= K, \quad K_i:= \beta_{v_i}(K_{i-1})\enspace\text{for }i=1,\dots,p, \enspace K_p=:\beta(K).$$
Then every boundary point of $\beta(K)$ is a diameter endpoint of $\beta(K)$, hence $\beta(K)$ is a body of constant width $1$ $($containing $K$, by Lemma $\ref{L1})$.

If also $L\in{\mathcal K}^\bullet$ and 
$$ \delta(K,L) \le\varepsilon\le\varepsilon_0,$$
then
$$ \delta(\beta_{v_1}(K),\beta_{v_1}(L))  \le \ell\varepsilon,\qquad \delta(\beta(K),\beta(L))  \le \ell\varepsilon.$$
\end{Lemma}

This is proved in \cite{MS13}. The first assertion is found on pp. 264--265 (the operator $\beta_0$ appearing there can now be chosen as the identity, since we have assumed that $K\in{\mathcal K}^\bullet$). The second assertion follows from Lemma 5 and the proof of Theorem 6 in \cite{MS13}.

We state an elementary geometric fact about balls, needed later to guarantee the quality of approximations.

Let $B_r$ be a ball of radius $r$, $0<r<1$, let $y\in{\rm bd}\,B_r$ and let $v$ be a unit tangent vector to $B_r$ at $y$. Let $B_1$ be a unit ball such that $B_r\subset B_1$ and $y+\lambda v\in{\rm bd}\,B_1$ for some $\lambda >0$. Let $u_1$ be the outer unit normal vector of $B_1$ at $y+\lambda v$. Then it is easy to see that there exists a number $\alpha(r)>0$ with the property that
\begin{equation}\label{B6}
\lambda \le \alpha(r) \enspace \Rightarrow \enspace \langle v,u_1\rangle \le 1/4.
\end{equation}

\section{Proof of Theorem \ref{T4}}\label{sec4.1}

After all these preparations, we can finish the proof of Theorem \ref{T4}. Let $W_0\in{\mathcal W}^n_1$, $v\in {\mathbb S}^{n-1}$ and $\varepsilon >0$ be given, without loss of generality $\varepsilon < \varepsilon_0$, where $\varepsilon_0$ is the number appearing in Lemma \ref{L2}. In a first step, we replace $W_0$ by the Minkowski combination
$$ W:= (1-\sigma)W_0+\sigma \frac{1}{2}B^n,$$
where we choose $0<\sigma<1$ so small that
\begin{equation}\label{E1} 
\delta(W,W_0)\le\varepsilon/2.
\end{equation}
The body $W$ is again of constant width $1$. It has the ball $(\sigma/2)B^n$ as a summand; in particular, it is smooth.

We define $\varepsilon'$ via
\begin{equation}\label{E2} 
\ell\varepsilon':= \min\{\varepsilon/2, \alpha(\sigma/2)\},
\end{equation}
where $\ell$ is given in Lemma \ref{L2} and $\alpha(\cdot)$ appears in (\ref{B6}).

In a second step, we change $W$ to a convex body nearby, which is still of diameter $1$, but no longer of constant width. We choose a finite set $M\subset \tau(W,S(-v))$ and define the convex body
$$ K:= {\rm conv}(M\cup \tau( W,{\rm cl}\,S(v))\}.$$
Then $K\subset W$, but $K$ has still diameter $1$, since it contains a pair of antipodal points of $W$ with outer normal vectors $\pm u\perp v$. By a proper choice of $M$ we achieve that
\begin{equation}\label{B1}
\delta(W,K) < \varepsilon'.
\end{equation}
Since $r_W>c_n$, we can further have $r_K>c_n$ and thus
\begin{equation}\label{B2}
K\in{\mathcal K}^\bullet.
\end{equation}

Let $v_1,\dots,v_p$ be the vectors provided by Lemma \ref{L2}. Since the assumptions and assertions of that lemma are invariant under rotations, we can assume, without loss of generality, that $v_1=v$. The body $\beta(K)$ constructed according to that lemma is a body of constant width $1$, and since $W$ is of constant width $1$, we have $\beta_v(W) =\beta(W)=W$. According to (\ref{B1}), Lemma \ref{L2} and (\ref{E2}), we get
\begin{equation}\label{E7} 
\delta(W,\beta_v(K)) \le \ell\varepsilon' \le \alpha(\sigma/2)
\end{equation}
and
\begin{equation}\label{E8} 
\delta (W,\beta(K)) \le \ell\varepsilon'\le \varepsilon/2.
\end{equation}
Together with (\ref{E1}), the latter gives 
\begin{equation}\label{E9}
\delta(W_0,\beta(K))\le\varepsilon.
\end{equation}

By Lemma \ref{L1}, the body $\beta_v(K)$ is $(K,v)$-complete (with $v=v_1$). We state that
\begin{equation}\label{B5} 
\tau(\beta_v(K),C(v)) \subset {\rm bd}\,\beta(K).
\end{equation}
For the proof, let $x\in \tau(\beta_v(K),C(v))$. Then there is an outer unit normal vector $u$ of $\beta_v(K)$ at $x$ with $\langle u,v \rangle \ge 1/2$. We assert that $x\in{\rm int}\,Z^+(K,v)$. Otherwise, $x$ is in the boundary of $Z^+(K,v)$, and in fact in the boundary of the cylinder $\{y+\lambda v:y\in K,\,\lambda\in{\mathbb R}\}$. Since $\langle u,v\rangle>0$, $x$ is of the form $x=y+\lambda v$ with $\lambda>0$ and a point $y\in {\rm bd}\,K$, in fact $y\in {\rm bd}\,W$. Since the ball $(\sigma/2) B^n$ is a summand of $W$, there is a translate $B_\sigma$ of this ball contained in $W$ with $y\in B_\sigma$. The distance of $x$ from $W$ is at most $\alpha(\sigma/2)$ by (\ref{E7}), hence $\lambda \le \alpha (\sigma/2)$. The point $x$ is not a boundary point of $W$ (since $y\in {\rm bd}\,W$ and $W$ is strictly convex), hence the normal cone of $\beta_v(K)$ at $x$ is the positive hull of finitely many vectors $u_j$, where $u_j$ is the outer unit normal vector at $x$ of some sphere ${\mathbb S} (z_j,1)$ with $z_j\in M$ and $x\in{\mathbb S}(z_j,1)$. It follows from (\ref{B6}) that $\langle u_j,v\rangle \le 1/4$. Since $u$ is a positive linear combination of the vectors $u_j$, we deduce that $\langle u,v\rangle \le 1/4$. This is a contradiction, which proves that $x\in{\rm int}\,Z^+(K,v)$. But then $x$ is a $(K,v)$-directed boundary point of $\beta_v(K)$. Since $\beta_v(K)$ is $(K,v)$-complete, the point $x$ is a diameter endpoint of $\beta_v(K)$. By Lemma \ref{L1} and induction, it is then also a diameter endpoint, in particular a boundary point, of $\beta(K)$. This completes the proof of (\ref{B5}).

By (\ref{B5}) and the construction of $\beta_v(K)$, $\beta(K)$ we have
\begin{equation}\label{B4}
\tau(\beta(K),C(v)) = \tau(\beta_v(K),C(v))  \subset \bigcup_{z\in  M}  {\mathbb S}(z,1).
\end{equation}
Now let $u\in C(v)$, and let $x\in \tau(\beta(K),C(v))$ be the point where $u$ is attained as outer normal vector. There is at least one point $z\in M$ for which $x\in{\mathbb S}(z,1)$. Suppose, first, that there is only one such point. Then $\|x-z_i\|>1$ for all $z_i\in M\setminus\{z\}$. Therefore, there is a neighbourhood $U$ of $x$ such that $U\cap\bd \beta(K)=U\cap {\mathbb S}(z,1)$. This implies that
\begin{equation}\label{D0} 
\bar\varrho(\beta(K),u,t) = 1\qquad\text{for each } t\in u^\perp \cap\Sn.
\end{equation}
Now suppose that $x\in{\mathbb S}(z_i,1)$ for (at least) two distinct points $z_i\in M$, $i=1,2$.  Then each of the vectors $x-z_1$, $x-z_2$ is an outer normal vector of $\beta(K)$ at $x$. At least one of them, say $x-z_2$, is linearly independent from $u$, and $\langle u,x-z_2\rangle >0$. Then there is a unit vector $t\perp u$ such that $x-z_2\in{\rm pos}\{u,t\}$, which implies that
$$ \bar\varrho(\beta(K),u,t)=0.$$
This completes the proof of Theorem \ref{T4} (with $Q=\beta(K)$).\qed

\section{Proofs of the main results}\label{sec5}

To prove our main results, we need one more piece of preparation. Let $K\in{\mathcal K}^n$. For $u\in\Sn$ and $o\not= t\perp u$ we write $x:=x_K(u)$ and define the two-dimensional plane
$$ L(K,u,t):= x+{\rm lin}\{u,t\}$$
and the two-dimensional closed halfplane
$$ D(K,u,t) :=\{x+\lambda u+\mu t:\lambda\in\R,\,\mu\ge 0\}.$$
Let
$$ P(K,u,t):= D(K,u,t)\cap {\rm relbd}\,(K|L(K,u,t)),$$
where $|$ denotes orthogonal projection.

For $\alpha>0$, we say that $K$ is $\alpha$-{\em wide at} $(u,t)$ if
$$ P(K,u,t)\cap {\rm int}\,B(x-\alpha u,\alpha) =\emptyset.$$

In the following lemma, we may write $\rho_i(K,u,t)$ for $\rho_i(K,x,u,t)$, since each $K\in {\mathcal W}^n_1$ is strictly convex.

\begin{Lemma}\label{L3}
Let $K\in {\mathcal W}^n_1$, let $u\in\Sn$ and $t\in u^\perp\cap\Sn$.

\noindent$\rm (a)$ $\bar\varrho_i(K,u,t)>0$ holds if and only if $K$ is $\alpha$-wide at $(u,t)$ for some $\alpha>0$.

\noindent$\rm (b)$ $\bar\varrho_s(K,u,t)<1$ holds if and only if $K$ is $\alpha$-wide at $(-u,-t)$ for some $\alpha>0$.

\noindent$\rm (c)$ $\varrho_i(K,u,t)>0$ for all $t\in u^\perp\cap\Sn$ holds if and only if $B(x_K(u)-\alpha u,\alpha)\subset K$ for some $\alpha>0$.
\end{Lemma}

\noindent{\em Proof.} Let $x:=x_K(u)$. Suppose that $\bar\varrho_i(K,u,t)>r>0$. Then there is some $\alpha$, $0<\alpha<r$, such that
$$ P(K,u,t)\cap B(x,2\alpha)\cap \inn B(x-r u,r)=\emptyset.$$
It follows that $P(K,u,t)\cap \inn B(x-\alpha u,\alpha)=\emptyset$, thus $K$ is $\alpha$-wide at $(x,t)$. The converse is clear.

Assertion (b) follows from (a) and the relation 
\begin{equation}\label{1}
\bar\varrho_s(K,u,t) + \bar\varrho_i(K,-u,-t) =1.
\end{equation}
In the plane, this was proved by Zamfirescu \cite{Zam94}. The general case follows from this, since tangential radii of curvature are defined via orthogonal projections to two-dimensional planes, and constant width $1$ is preserved under such projections. 

Alternatively, (\ref{1}) follows from the representation of upper and lower tangential radii of curvature in terms of the support function (see (\ref{2.2.a}) and (\ref{2.2.b})), namely
$$ \bar\varrho_s(K,u,t)= \limsup_{\beta\downarrow 0} \frac{h(K,w_\beta)-\langle x,w_\beta\rangle}{1-\langle u,w_\beta\rangle}, \quad w_\beta= u\cos\beta +t\sin\beta,$$
and a similar relation for $\bar\varrho_i$. Since $K$ is of constant width $1$, we have  $h(K,w)+h(K,-w)=1$ for any unit vector $w$ and $x_K(-u)= x_K(u)-u$, which gives
$$\frac{h(K,w)-\langle x_K(u),w\rangle}{1-\langle u,w\rangle} +
\frac{h(K,-w)-\langle x_K(-u),-w\rangle}{1-\langle u,w\rangle} =1$$
and hence (\ref{1}).

To prove (c), suppose that $\varrho_i(K,u,t)>0$ for $t\in u^\perp\cap\Sn$. For given $t\in u^\perp\cap\Sn$, it follows from (\ref{9a}) that there are positive numbers $\eta_t$ and $r_t$ such that $ \rho({\mathcal D}_x(K,\eta),t) \ge r_t$, and hence $r_tt\in {\mathcal D}_x(K,\eta)$, for $0<\eta\le \eta_t$. We choose $t_1,\dots,t_n\in u^\perp\cap\Sn$ which positively span $u^\perp$. Then there are numbers $\eta_0,r_0>0$ such that $ \rho({\mathcal D}_x(K,\eta),t_i) \ge r_0$ for $0<\eta\le \eta_0$ and $i=1,\dots,n$. We have
$$ o\in{\rm int\,conv}\{r_0t_1,\dots,r_0t_n\}\subset {\mathcal D}_x(K,\eta)$$
for $0<\eta\le\eta_0$, hence there is a number $r>0$ such that
$$ B_u(o,r)\subset {\mathcal D}_x(K,\eta), \qquad B_u(o,r):= B(o,r)\cap u^\perp. $$
By (\ref{2.5.3d}), this implies that
$$ \sqrt{2\eta} B_u(o,r)+x-\eta u\subset K$$
for $0<\eta\le\eta_0$. It follows that, in a neighbourhood of $x$, the body $K$ contains a paraboloid of revolution which contains $x$. From this, it follows that $B(x-\alpha u,\alpha)\subset K$ for some $\alpha>0$. The converse assertion of (c) is clear. 
\qed

Now we are in a position to prove a Baire category result, from which Theorems \ref{T1} and \ref{T2} can be deduced.

\begin{Theorem}\label{T3}
A typical convex body $K$ of constant width $1$ in $\R^n$ has the following property. Let $(x,u)$ be a support element of $K$; then $(x,u)$ is regular and either
$$ \bar\varrho_s(K,u,t) =1\qquad \text{for all } t\in u^\perp\cap\Sn $$
or
$$ \varrho_i(K,u,t)=0\qquad \text{for at least one }t\in u^\perp\cap\Sn. $$
\end{Theorem}

\noindent{\em Proof.} First we remark that a convex body of constant width is strictly convex, and a typical convex body of constant width is smooth  (the proof is the same as for general convex bodies; see \cite[Thm. 2.7.1]{Sch14}); hence it has only regular support elements.

We define the set
\begin{eqnarray*}
{\mathcal A} &:=& \big\{K\in {\mathcal W}^n_1: \text{For each $u\in \Sn$,}\\
& & \text{either }\bar\varrho_s(K,u,t)=1\text{ for all } t\in u^\perp\cap\Sn\\
& &  \text{or } \varrho_i(K,u,t)=0 \text{ for some }t\in u^\perp\cap\Sn\big\}.
\end{eqnarray*}
Then we choose finitely many points $p_1,\dots,p_{m_0}\in {\mathbb S}^{n-1}$ such that $\bigcup_{m=1}^{m_0} C(p_m)= {\mathbb S}^{n-1}$, and for $k\in{\mathbb N}$ and $m=1,\dots,m_0$ we define
\begin{eqnarray*}
{\mathcal B}_{k,m} &:=& \big\{K\in {\mathcal W}^n_1: \text{There exists $u\in C(p_m)$ such that}\\
& & \text{there exists $t\in u^\perp\cap\Sn$ such that $K$ is $k^{-1}$-wide at $(-u,-t)$}\\
& &  \text{and } B(x_K(u)-k^{-1}u,k^{-1})\subset K\big\}.
\end{eqnarray*}

Let $K\in{\mathcal W}^n_1\setminus{\mathcal A}$. Then there exists a vector $u\in\Sn$ such that 
\begin{equation}\label{6.1} 
\bar\varrho_s(K,u,t)<1\qquad\text{for some }t\in u^\perp\cap\Sn
\end{equation}
and
\begin{equation}\label{6.2}  
\varrho_i(K,u,t)>0\qquad \text{for all } t\in u^\perp\cap\Sn.
\end{equation}

The vector $u$ is contained in a suitable set $C(p_m)$. By Lemma \ref{L3}, (\ref{6.1}) implies that $K$ is $\alpha$-wide at $(-u,-t)$ for some $\alpha>0$. Also by Lemma \ref{L3}, (\ref{6.2}) implies that $B(x_K(u)-\alpha u,\alpha)\subset K$ for some $\alpha>0$. Thus, $K\in{\mathcal B}_{k,m}$ for suitable $k\in{\mathbb N}$, $m\in\{1,\dots,m_0\}$. Conversely, if $K\in {\mathcal B}_{k,m}$, then $K\notin{\mathcal A}$. Thus
\begin{equation}\label{com} 
{\mathcal W}^n_1 \setminus {\mathcal A}=\bigcup _{k\in{\mathbb N}}\bigcup_{m=1}^{m_0} {\mathcal B}_{k,m}.
\end{equation}

The strategy is now clear: we show that each set ${\mathcal B}_{k,m}$ is closed and without interior points, hence nowhere dense in ${\mathcal W}^n_1$.

\begin{Lemma}\label{L4}
For $k\in{\mathbb N}$, $m\in\{1,\dots,m_0\}$ the set ${\mathcal B}_{k,m}$ is closed.
\end{Lemma}

\noindent{\em Proof.} Let $(K_j)_{j\in{\mathbb N}}$ be a sequence in ${\mathcal B}_{k,m}$ that converges to a convex body $K$. Then $K$ is also of constant width $1$. For each $j\in{\mathbb N}$ we can choose a vector $u_j\in C(p_m)$ with the following properties. We can choose a vector $t_j\in u_j^\perp\cap\Sn$ such that $K_j$ is $k^{-1}$-wide at $(-u_j,-t_j)$, and $B(x_{K_j}(u_j)-k^{-1}u_j,k^{-1})\subset K_j$. Since $C(p_m)$ is compact, we can assume (after selecting a subsequence and changing the notation) that the sequence $(u_j)_{j\in{\mathbb N}}$ converges to some vector $u\in C(p_m)$.  Again choosing a subsequence and changing the notation, we can assume that also the sequence $(t_j)_{j\in {\mathbb N}}$ converges to a vector $t$.

We claim that $K$ is $k^{-1}$-wide at $(-u,-t)$. Suppose that this is false. Then there exist a point $z\in P(K,-u,-t)$ and a ball $B$ with center $z$ such that
\begin{equation}\label{4} 
B\subset\inn B(x_K(-u)+k^{-1}u,k^{-1}).
\end{equation}
Now $K_j\to K$ and $u_j\to u$ implies that $x_{K_j}(-u_j)\to x_K(-u)$ and hence
$$ L(K_j,-u_j,-t_j) \to L(K,-u,-t), \qquad D(K_j,-u_j,-t_j) \to D(K,-u,-t) \qquad\text{for }j\to\infty.$$
We conclude that
$$ K_j|L(K_j,-u_j,-t_j) \to K|L(K,-u,-t)\qquad\text{for }j\to\infty.$$
But then (\ref{4}) implies that
$$ P(K_j,-u_j,-t_j)\cap \inn B(x_{K_j}(-u_j)+k^{-1}u_j,k^{-1})\not=\emptyset$$
for sufficiently large $j$, which contradicts the fact that $K_j$ is $k^{-1}$-wide at $(-u_j,-t_j)$. Thus, $K$ is $k^{-1}$-wide at $(-u,-t)$.

Further, from $B(x_{K_j}(u_j)-k^{-1}u_j,k^{-1})\subset K_j$ it follows that $B(x_{K}(u)-k^{-1}u,k^{-1})\subset K$. Thus, $K\in {\mathcal B}_{k,m}$, which shows that the latter set is closed. \qed

\begin{Lemma}\label{L5}
For $k\in{\mathbb N}$, $m\in\{1,\dots,m_0\}$, the set ${\mathcal B}_{k,m}$ is nowhere dense.
\end{Lemma}

\noindent{\em Proof.} Since ${\mathcal B}_{k,m}$ is closed by Lemma \ref{L4}, it suffices to prove that ${\mathcal B}_{k,m}$ has empty interior. Let $W_0\in{\mathcal W}^n_1$ and a number $0<\varepsilon\le \varepsilon_0$ be given (the  number $\varepsilon_0$ appears in Lemma \ref{L2}). 

According to Theorem \ref{T4} (with $v=p_m$) there exists a convex body $Q\in{\mathcal W}^n_1$ with 
$$ \delta(W_0,Q)\le \varepsilon$$ 
and with the property that, for each $u\in C(p_m)$, either
$$ \bar\varrho(Q,u,t)=1\qquad\text{for all } t\in u^\perp\cap\Sn$$
or
$$ \bar\varrho(Q,u,t) = 0\qquad \text{for at least one }t\in u^\perp\cap \Sn.$$
By (\ref{3.19}), the latter implies
$$ \varrho(Q,u,t) = 0\qquad \text{for at least one }t\in u^\perp\cap \Sn.$$
By Lemma \ref{L3}, we obtain for each $u\in C(p_m)$ that either
$$ K \text{ is not $k^{-1}$-wide at $(-u,-t)$, for all }t\in u^\perp\cap \Sn$$
or $$ B(x_K(u)-k^{-1}u,k^{-1}) \not\subset K.$$
Thus, $Q\notin {\mathcal B}_{km}$. Since $W_0\in{\mathcal W}^n_1$ was arbitrary and $\varepsilon>0$ can be arbitrarily small, this shows that ${\mathcal B}_{k,m}$ has empty interior. \qed

Since ${\mathcal B}_{k,m}$ is nowhere dense, (\ref{com}) shows that ${\mathcal A}$ is comeagre in ${\mathcal W}^n_1$. This completes the proof of Theorem \ref{T3}.

\vspace{3mm}

\noindent{\em Proof of Theorem} \ref{T1}

\vspace{2mm}

A typical convex body $K$ in ${\mathcal W}^n_1$ has the property stated in Theorem \ref{T3}. As mentioned in Section \ref{sec3}, ${\mathcal H}^{n-1}$-almost all vectors $u\in \Sn$ are ordinary normal vectors of $K$. At an ordinary normal vector $u$, all tangential radii of curvature exist and are finite, hence by Lemma \ref{L0} the same holds for the sectional radii of curvature. Hence, at almost all $u$, either the curvature indicatrix is a unit ball, or at last one 
radius of curvature is zero. \qed

\vspace{3mm}

\noindent{\em Proof of Theorem} \ref{T2}

\vspace{2mm}

A typical convex body $K$ in ${\mathcal W}^n_1$ has the property stated in Theorem \ref{T3}. As mentioned in Section \ref{sec3}, ${\mathcal H}^{n-1}$-almost all boundary points of $K$ are normal points. Let $x$ be a normal boundary point of $K$, and let $u$ be the outer unit normal vector to $K$ at $x$ (it is unique for typical convex bodies in ${\mathcal W}^n_1$). For $t\in u^\perp\cap\Sn$ we have $\kappa(K,u,t)=1/\varrho(K,u,t)$ and $\varrho(K,u,t) \le \bar \varrho_s(K,u,t)$ by (\ref{3.19}), further $\bar\varrho_s(K,u,t)\le 1$ since $K$ is of constant width $1$. Hence, the curvatures of $K$ at $x$ are positive. Since $x$ is a normal point, they are also finite. Now it follows from Lemma \ref{L0} that at $u$ all the tangential radii of curvature exist and are positive and finite. Since $K$ has the property stated in Theorem \ref{T3}, they can only be $1$. This means (also by Lemma \ref{L0}) that the curvature indicatrix of $K$ at $x$ is a unit ball, hence all radii of curvature at $x$ are equal to $1$.

A body of constant width is strictly convex, and a typical convex body of constant width $1$ is smooth. Hence, we can assume that $K$ is smooth and strictly convex. Therefore, the antipodal map $\psi$ of $K$, which associates with every boundary point $x$ of $K$ with outer unit normal vector $u$ the unique boundary point $\psi(K)$ with outer normal vector $-u$, is a homeomorphism. It maps the set ${\rm bd}_nK$ of normal boundary points, which has full measure and hence is uncountable and dense in $\bd K$, to a set which is also uncountable and dense in $\bd K$. Let $x\in {\rm bd}_nK$, let $u$ be the outer unit normal vector at $x$ and let $t$ be a unit tangent vector at $x$. As shown above, $\bar\varrho(K,u,t)=1$. From (\ref{1}) (and the existence of the tangential radii of curvature at $u$) it follows that $\bar\varrho(K,-u,-t)=0$. Thus, all radii of curvature at $\psi(x)$ are zero, from which the last assertion of Theorem \ref{T2} follows.  \qed

\section{Appendix}

Here we provide an auxiliary result on convergence of convex sets, which has been used in the proof of Lemma \ref{L0}. For non-empty closed convex sets $K_i,K\subset\R^n$, $i\in{\mathbb N}$, the limit relation s-$\lim_{i\to\infty} K_i=K$ is defined by
$$ \lim_{i\to \infty} h(K_i,u)= h(K,u)\qquad\text{for }u\in\Sn,$$
where $h$ denotes the support function, and if $o\in K_i,K$, then the limit relation r-$\lim_{i\to\infty} K_i=K$ is defined by
$$ \lim_{i\to \infty} \rho(K_i,u)= \rho(K,u)\quad\text{for }u\in\Sn,$$
where $\rho$ denotes the radial function. In the general case, where the sets $K_i$ need not be bounded or $o$ need not be an interior point of $K$, none of these types of convergence implies the other. This is different for convex bodies containing $o$ in the interior.

In the proof of the following lemma,  we denote by $K^\circ$ the polar body of a convex body $K$ with $o\in \inn K$, and we make use of the fact that
\begin{equation}\label{7.0} 
\rho(K,u) =\frac{1}{h(K^\circ,u)}\qquad\text{for }u\in \Sn,
\end{equation}
see \cite[(1.52)]{Sch14}. 

\vspace{3mm}

\begin{Lemma}\label{L7.1}
Let $K_i,K \in \K^n$ and suppose that $o\in \inn K$ and $o\in K_i$ for $i\in{\mathbb N}$. Then the relations
\begin{equation}\label{7.1}  
\text{\rm s-}\lim_{i\to\infty} K_i=K
\end{equation}
and
\begin{equation}\label{7.2}  
\text{\rm r-}\lim_{i\to\infty} K_i=K
\end{equation}
are equivalent.
\end{Lemma}

\begin{proof}  Suppose that (\ref{7.1}) holds. Then $K_i\to K$ in the Hausdorff metric, by \cite{Sch14}, Thm. 1.8.15 and Lemma 1.8.14. Since $o\in\inn K$, there are numbers $r,R>0$ such that $B(o,r)\subset K_i\subset B(o,R)$ for almost all $i$. For these $i$, the polar sets $K^\circ_i$ are convex bodies satisfying $B(o,1/R)\subset K^\circ_i\subset B(o,1/r)$. On the space of convex bodies containing $o$ in the interior, the polarity mapping is continuous, hence we conclude that $K_i^\circ\to K^\circ$ in the Hausdorff metric, thus $\text{s-}\lim_{i\to\infty} K_i^\circ=K^\circ$, which implies  $\text{r-}\lim_{i\to\infty} K_i=K$ by (\ref{7.0}). Thus (\ref{7.2}) holds. 

Suppose that (\ref{7.2}) holds. Let fixed vectors $x_1,\dots,x_m\in {\mathbb S}^{n-1}$ be given. Since $\lim_{i\to\infty} \rho(K_i,\cdot)>0$, there is a number $\varepsilon>0$ such that $\rho(K_i,x_r)\ge \varepsilon$ for $r=1,\dots,m$ and for all sufficiently large $i$. If $x_1,\dots,x_m$ are suitably chosen, this together with the convexity of the sets $K_i$ implies the existence of a number $r>0$ such that $B(o,r)\subset K_i$ for almost all $i$. Now (\ref{7.2}) and (\ref{7.0}) yield that $\text{s-}\lim_{i\to\infty} K_i^\circ=K^\circ$. As already shown, this implies $\text{r-}\lim_{i\to\infty} K_i^\circ=K^\circ$, which in turn implies (\ref{7.1}).
\end{proof}

\noindent Authors' addresses:\\[2mm]
Imre B\'{a}r\'{a}ny\\
R\'{e}nyi Institute, Hungarian Academy of Sciences\\
POB 127, 1364 Budapest, Hungary\\[1mm]
and\\[1mm]
Department of Mathematics, University College London\\
Gower Street, London, WC1E 6BT, UK\\
E-mail: barany@renyi.hu\\[3mm]
Rolf Schneider\\
Mathematisches Institut, Albert-Ludwigs-Universit{\"a}t\\
D-79104 Freiburg i. Br., Germany\\
E-mail: rolf.schneider@math.uni-freiburg.de


\begin{thebibliography}{99}

\bibitem{Ale39} Aleksandrov,  A. D.,  Almost everywhere existence of the second differential of a convex function and some properties of convex surfaces connected with it (in Russian). \textit{Uchenye Zapiski Leningrad. Gos. Univ., Math. Ser.} \textbf{6} (1939), 3--35.

\bibitem{Bla56} Blaschke, W., \emph{Kreis und Kugel}. 2nd edn., W. de Gruyter, Berlin, 1956 (1st edn: 1916).

\bibitem{BV10} Borwein, J. M., Vanderwerff, J. D., {\em Convex Functions: Constructions, Characterizations and Counterexamples.} Cambridge University Press, Cambridge, 2010.

\bibitem{Buc36} B\"uckner, H., \"Uber Fl\"achen von fester Breite. {\em Jahresber. Deutsche Math.-Verein.} {\bf 46} (1936), 96--139.

\bibitem{Bus58} Busemann, H., \emph{Convex Surfaces}. Interscience, New York, 1958.

\bibitem{BF36} Busemann, H., Feller, W., Kr\"{u}mmungseigenschaften konvexer Fl\"{a}chen. \emph{Acta Math}. \textbf{66} (1936), 1--47.

\bibitem{CG83} Chakerian, G. D., Groemer, H., Convex bodies of constant width. In \emph{Convexity and its Applications} (P. M. Gruber, J. M. Wills, eds.), pp. 49--96, Birkh\"{a}user, Basel, 1983. 

\bibitem{Jes29} Jessen, B., Om konvekse Kurvers Krumning. \textit{Mat. Tidsskr.} B (1929), 50--62.

\bibitem{MS13} Moreno, J. P., Schneider, R., Lipschitz selections of the diametric completion mapping in Minkowski spaces. {\em Adv. Math.} {\bf 233} (2013), 248--267.

\bibitem{Sch14} Schneider, R., \textit{Convex Bodies -- The Brunn--Minkowski Theory.} 2nd ed., Cambridge University Press, Cambridge, 2014.

\bibitem{Sch14b} Schneider, R., Curvatures of typical convex bodies---the complete picture. {\em Proc. Amer. Math. Soc.} (in press).

\bibitem{Zam80} Zamfirescu, T., {\em The curvature of most convex surfaces vanishes almost everywhere.} Math. Z. \textbf{174} (1980), 135--139.

\bibitem{Zam94} Zamfirescu, T., On the curvatures of convex curves of constant width. {\em Atti Sem. Mat. Fis. Univ. Modena} {\bf 52} (1994), 253--256.

\end{thebibliography}
\end{document}